\documentclass[11pt]{amsart}
\usepackage{geometry}                
\usepackage{graphicx}
\usepackage{amssymb}
\usepackage{epstopdf}
\newtheorem{theorem}{Theorem}[section]
\newtheorem{lemma}{Lemma}[section]
\newtheorem{corollary}{Corollary}[section]

\newtheorem{conjecture}{Conjecture}[section]
\newtheorem{definition}{Definition}[section]

\newtheorem{remark}{Remark}[section]
\newtheorem{question}{Question}[section]
\numberwithin{equation}{section}

\def\Z{\mathbb Z}

\def\R{\mathbb R}

\def\N{\mathbb N}

\def\d{\partial}

\def\a{\alpha}

\def\g{\gamma}
\def\l{\lambda}

\title{Algebras of smooth functions and holography of traversing flows}

 \author[G.~Katz]{Gabriel Katz}
\address{MIT, Department of Mathematics, 77 Massachusetts Ave., Cambridge, MA 02139, U.S.A.}
\email{gabkatz@gmail.com}

\begin{document}

\maketitle 

\begin{abstract} Let $X$ be a smooth compact manifold and $v$ a vector field on $X$ which admits a smooth function $f: X \to \R$ such that $df(v) > 0$. Let $\partial X$ be the boundary of $X$. We denote by $C^\infty(X)$ the algebra of smooth functions on $X$ and by $C^\infty(\partial  X)$ the algebra of smooth functions on $\partial  X$. With the help of $(v, f)$, we introduce two subalgebras $\mathcal A(v)$ and $\mathcal B(f)$ of $C^\infty(\partial  X)$ and prove  (under mild hypotheses) that $C^\infty(X) \approx \mathcal A(v) \hat\otimes \mathcal B(f)$, the topological tensor product. Thus the topological algebras $\mathcal A(v)$ and $\mathcal B(f)$, \emph{viewed  as boundary data}, allow for a reconstruction of $C^\infty(X)$. As a result, $\mathcal A(v)$ and $\mathcal B(f)$ allow for the recovery of the smooth topological type of the bulk $X$.
\end{abstract}

\section{Introduction}

It is classically known that the normed algebra $C^0(X)$ of continuous real-valued functions on a compact space $X$ determines its topological type \cite{GRS}, \cite{Ga}, \cite{Br}. In this context, $X$ is interpreted as the space of maximal ideals of the algebra $C^0(X)$. 
In a similar spirit, the algebra $C^\infty(X)$  of smooth functions on a compact smooth manifold $X$ (the algebra $C^\infty(X)$ is considered in the Whitney topology \cite{W3}) determines the \emph{smooth} topological type of $X$ \cite{KMS}, \cite{Na}. Again, $X$ may be viewed as the space of maximal ideals of the algebra $C^\infty(X)$. \smallskip

Recall that a harmonic function $h$ on a compact connected Riemannian manifold $X$ is uniquely determined by its restriction to the smooth boundary $\d X$ of $X$. In other words, the Dirichlet  boundary value problem has a unique solution in the space of harmonic functions. Therefore, the vector space $\mathcal H(X)$ of harmonic functions on $X$ is rigidly determined by its restriction (trace) $\mathcal H^\d(X) := \mathcal H(X)|_{\d X}$ to the boundary $\d X$. As we embark on our journey,  this fact will serve us as a beacon.  \smallskip

This paper revolves around the following question: 

{\sf Which algebras of smooth functions on the boundary $\d X$ can be used to reconstruct the algebra $C^\infty(X)$ and thus the smooth topological type of $X$?} \smallskip

\noindent Remembering the flexible nature of smooth functions (in contrast with the rigid harmonic ones), at the first glance, we should anticipate the obvious answer "None!".  However, 
when $X$ carries an additional geometric structure, then the question, surprisingly, may have  a positive answer. The geometric structure on $X$ that does the trick is a vector field (i.e., an ordinary differential equation), drawn from a massive class of vector fields which we will introduce below.\smallskip

Let $X$ be a compact connected smooth $(n+1)$-dimensional manifold with boundary and $v$ a smooth vector field admitting a {\sf Lyapunov function} $f: X \to \R$ so that $df(v) > 0$. We call such vector fields {\sf traversing}. We assume that $v$ is in general position with respect to the boundary $\d X$  and call such vector fields {\sf boundary generic} (see \cite{K1} or \cite{K3}, Definition 5.1, for the notion of {\sf boundary generic} vector fields). Temporarily, it will be sufficient to think of the boundary generic vector fields $v$ as having only $v$-trajectories that are tangent to the boundary $\d X$ with the order of tangency less than or equal to $\dim(X)$. Section 3 contains a more accurate definition. \smallskip

Informally, we use the term ``{\sf holography}" when some residual structures on the boundary $\d X$ are sufficient for a reconstruction of similar structures on the bulk $X$. \smallskip

Given such a triple $(X, v, f)$, in Section 3, we will introduce two subalgebras, $\mathcal A(v) = C^\infty(\d X, v)$ and $\mathcal B(f) = (f^\d)^\ast(C^\infty(\R))$, of the algebra $C^\infty(\d X)$, which depend only on $v$ and $f$, respectively. By Theorem \ref{th.main_alg}, $\mathcal A(v)$ and $\mathcal B(f)$ will allow for a reconstruction of the algebra $C^\infty(X)$.  In fact, the boundary data, generated by these subalgebras, lead to a unique (rigid) ``solution" $$C^\infty(X) \approx C^\infty(\d X, v) \, \hat{\otimes}\, (f^\d)^\ast(C^\infty(\R)),$$ the topological tensor product of the two algebras. As a result, the pair $\mathcal A(v)$, $\mathcal B(f)$, ``residing on the boundary", determines the smooth topological type of the bulk $X$ and of the $1$-dimensional foliation $\mathcal F(v)$, generated by the $v$-flow. 
 
\section{Holography on manifolds with boundary and the causality maps}

Let $X$ be a compact connected smooth $(n+1)$-dimensional manifold with boundary $\d_1 X =_{\mathsf{def}} \d X$ (we use this notation for the boundary $\d X$ to get some consistency with  similar notations below), and $v$ a smooth traversing vector field, admitting a smooth {\sf Lyapunov function} $f: X \to \R$.  We assume that $v$ is {\sf boundary generic}.

We denote by $\d_1^+X(v)$ the subset of $\d_1 X$ where $v$ is directed inwards of $X$ or is tangent to $\d_1 X$. Similarly, $\d_1^-X(v)$ denotes the subset of $\d_1 X$ where $v$ is directed outwards of $X$ or is tangent to $\d_1 X$. \smallskip

Let $\mathcal F(v)$ be the $1$-dimensional {\sf oriented foliation}, generated by the traversing $v$-flow.\smallskip

We denote by $\g_x$ the $v$-trajectory through $x \in X$. Since $v$ is traversing and boundary generic, each $\g_x$ is homeomorphic either a closed segment, or to a singleton \cite{K1}.\smallskip

In what follows, we embed the compact manifold $X$ in an open manifold $\hat X$ of the same dimension so that $v$ extends to a smooth vector field $\hat v$ on $\hat X$,  $f$ extends to a smooth function $\hat f$ on $\hat X$, and $d\hat f(\hat v) > 0$ in $\hat X$. We treat $(\hat X, \hat v, \hat f)$ as a germ in the vicinity of $(X, v,  f)$.

\begin{definition}\label{def.property_A} 
We say that a boundary generic and traversing vector field $v$ possesses  {\sf  Property $\mathsf A$}, if each $v$-trajectory $\g$ is either transversal to $\d_1 X$ at {\sf some} point of the set $\g \cap \d_1 X$, or $\g \cap \d_1 X$ is a singleton $x$ and $\g$ is quadratically tangent to $\d_1 X$ at $x$. \hfill $\diamondsuit$
\end{definition}
 \smallskip

A traversing vector field $v$ on $X$ induces a structure of a {\sf partially-ordered set} $(\d_1 X, \succ_v)$ on the boundary $\d_1 X$: for $x, y \in \d_1 X$, we write $y \succ x$ if the two points lie on the same $v$-trajectory $\g$ and  $y$ is reachable from $x$ by moving in the $v$-direction. \smallskip
  
We denote by $\mathcal T(v)$ the {\sf trajectory space} of $v$ and by $\Gamma: X \to \mathcal T(v)$ the obvious projection. For a traversing and boundary generic $v$, $\mathcal T(v)$ is a compact space in the topology induced by $\Gamma$. 
Since any trajectory of a traversing $v$ intersects the boundary $\d_1 X$, we get that $\mathcal T(v)$ is a quotient of $\d_1 X$ modulo the partial order relation $ \succ_v$. 

\begin{figure}[ht]\label{fig1.4}
\centerline{\includegraphics[height=2.3in,width=3in]{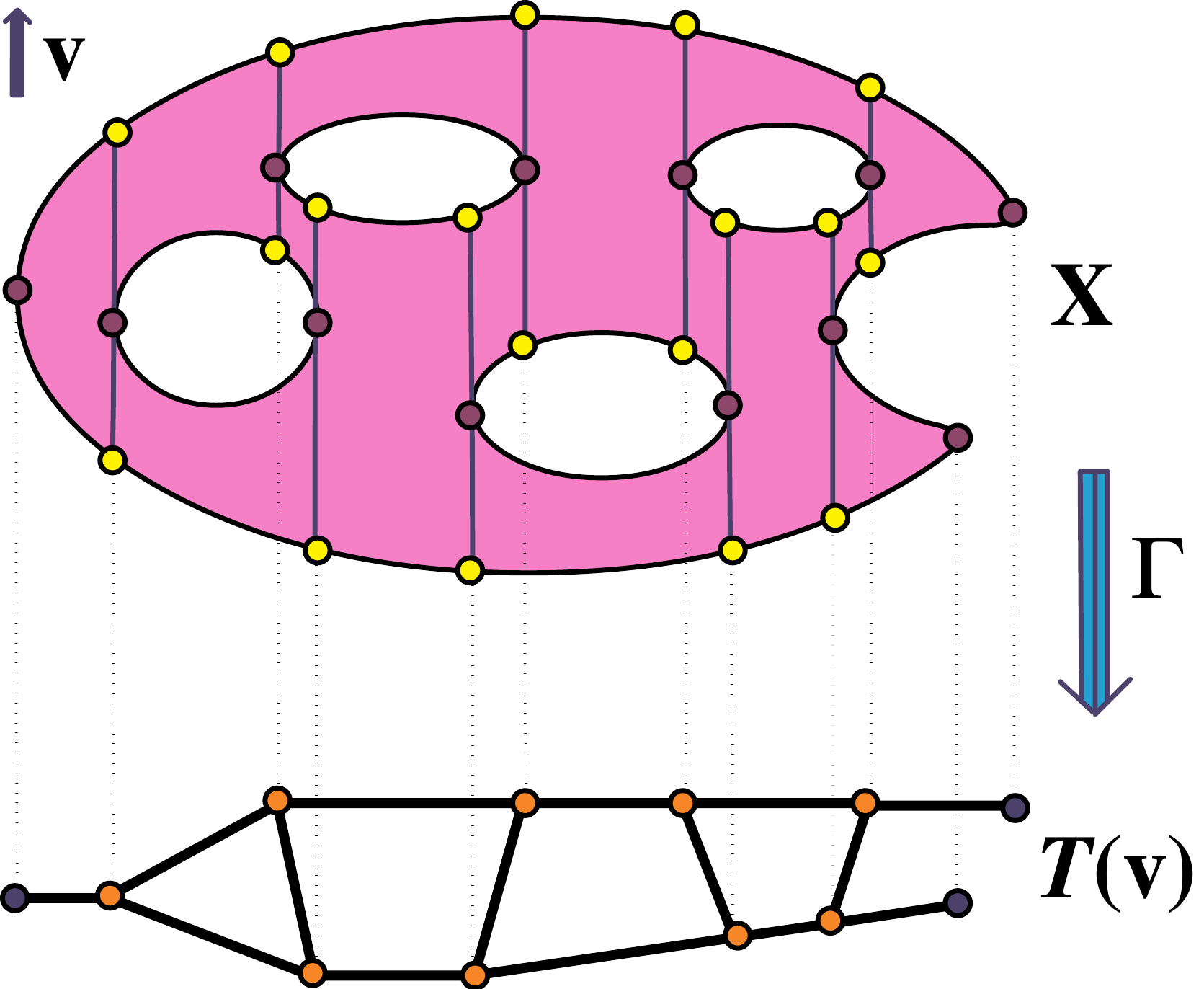}}
\bigskip
\caption{\small{The map $\Gamma: X \to \mathcal T(v)$ for a traversally generic (vertical) vector field $v$ on a disk with $4$ holes. The trajectory space is a graph whose verticies are of valencies $1$ and $3$. The restriction of $\Gamma$ to $\d_1X$ is a surjective map $\Gamma^\d$ with finite fibers of cardinality $3$ at most; a generic fiber has cardinality $2$.}}
\end{figure}
\smallskip

A traversing and boundary generic $v$ gives rise to the {\sf causality (scattering) map}
\begin{eqnarray}
C_v: \d_1^+X(v) \to \d_1^-X(v)
\end{eqnarray}
that takes each point $x \in \d_1^+X(v)$ to the unique consecutive point $y \in \g_x \cap \d_1^-X(v)$ that can be reached from $x$ in the $v$-direction. If no such $y \neq x$ is available, we put $C_v(x) = x$.  We stress that typically $C_v$ is a \emph{discontinuous} map (see Fig. \ref{fig.1.7}). \smallskip

\begin{figure}[ht]\label{fig.1.7}
\centerline{\includegraphics[height=2.5in,width=4in]{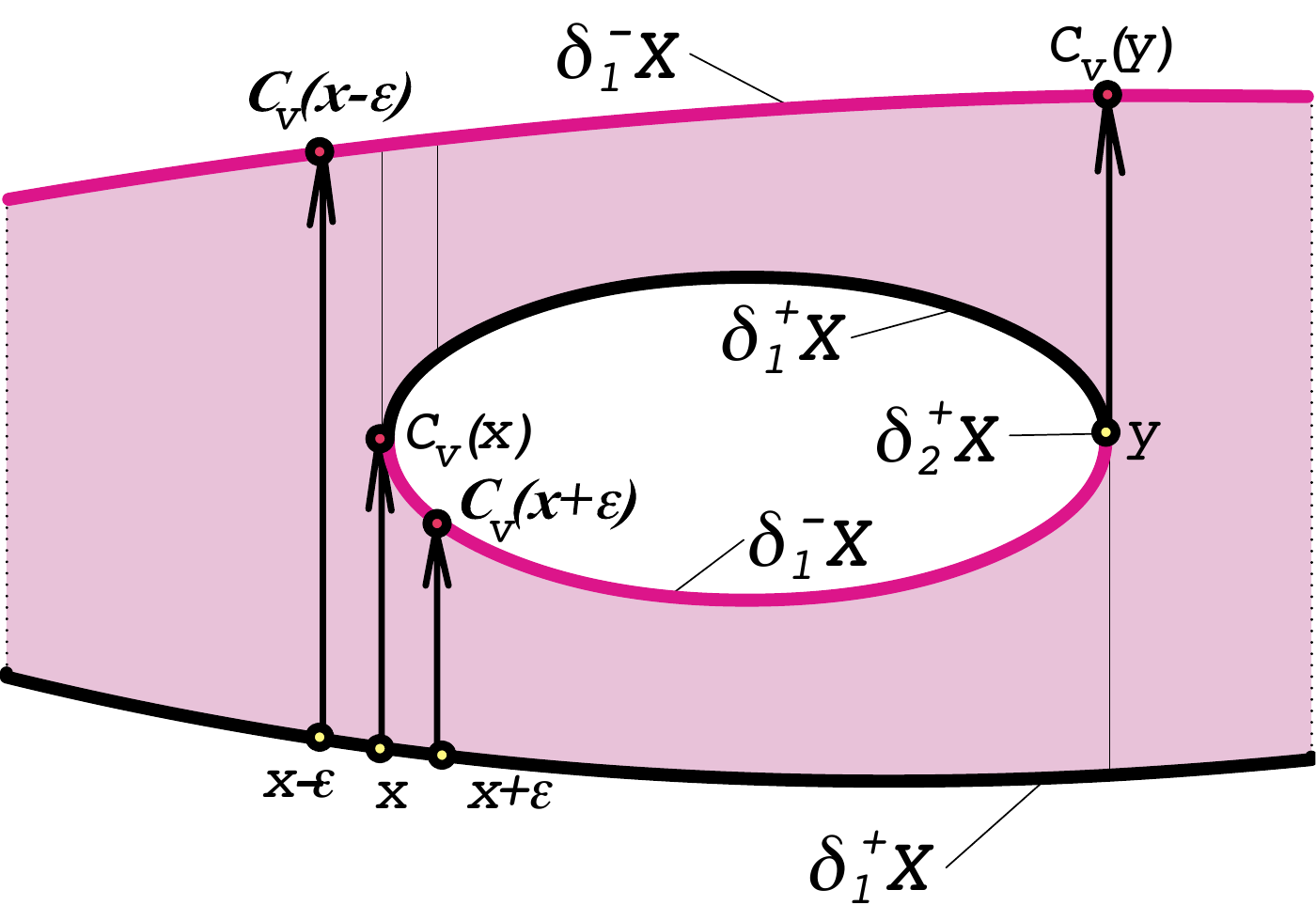}}
\bigskip
\caption{\small{An example of the causality map $C_v: \d_1^+X(v) \to \d_1^-X(v)$.} Note the essential discontinuity of $C_v$ in the vicinity of $x$. }
\end{figure}

We notice that, for any smooth positive function $\l: X \to \R_+$, we have $C_{\l\cdot v} = C_v$; thus the causality map depends only on the conformal class of a traversing vector field $v$. In fact, $C_v$ depends only on the oriented foliation $\mathcal F(v)$, generated by the $v$-flow.
\smallskip

In the paper, we will discuss two kinds of intimately related holography problems. The first kind amounts to the question: {\sf To what extend given boundary data are sufficient for reconstructing the unknown bulk and the traversing $v$-flow on it, or rather, the foliation $\mathcal F(v)$?} This question may be represented symbolically by the two diagrams:
 \begin{eqnarray}\label{3_Reconstruction Holo}
\bullet  \text{\sf Holographic Reconstruction Problem}\nonumber \\ 
(\d_1X, \; \succ_v, )\; \stackrel{\mathbf{??}}{\longrightarrow} \;
 (X,\; \mathcal F(v)),\\
 (\d_1X, \; \succ_v, \; f^\d)\; \stackrel{\mathbf{??}}{\longrightarrow} \;
 (X,\; \mathcal F(v),\; f),
 \end{eqnarray}
 where $\succ_{v}$ denotes the partial order on boundary, defined by the causality map $C_v$, and the symbol ``$\stackrel{\mathbf{??}}{\longrightarrow}$" points to the unknown ingredients of the diagrams.\smallskip

The second kind of problem is: {\sf Given two manifolds,  $X_1$ and $X_2$, equipped with traversing flows, and a diffeomorphism $\Phi^\d$ of their boundaries, respecting the relevant boundary data, is it possible to extend $\Phi^\d$ to a diffeomorphism/homeomorphism $\Phi: X_1 \to X_2$  that respects the corresponding flows-generated structures in the interiors of the two manifolds?} 

\noindent This problem may be represented by the commutative diagrams:
\begin{eqnarray}\label{3_Extension Holo}
\bullet  \text{\sf Holographic Extension Problem \quad} \nonumber \\ 
 (\d_1X_1, \; \succ_{v_1})\; \stackrel{\mathsf{inc}}{\longrightarrow} \; (X_1,\; \mathcal F(v_1))\nonumber \\ 
 \quad \downarrow \; \Phi^\d  \quad \quad  \quad \quad \quad  \quad \downarrow \; ?? \;\; \Phi \quad 
 \\
 (\d_1X_2, \; \succ_{v_2})\; \stackrel{\mathsf{inc}}{\longrightarrow} \; (X_2,\; \mathcal F(v_2)) \nonumber \\
 \nonumber \\
(\d_1X_1, \; \succ_{v_1},\; f_1^\d)\; \stackrel{\mathsf{inc}}{\longrightarrow} \; (X_1,\; \mathcal F(v_1), f_1)\nonumber \\ 
 \quad \downarrow \; \Phi^\d  \quad \quad  \quad \quad \quad  \quad \downarrow \; ?? \;\; \Phi \quad 
 \\
 (\d_1X_2, \; \succ_{v_2}, \; f_2^\d)\; \stackrel{\mathsf{inc}}{\longrightarrow} \; (X_2,\; \mathcal F(v_2),\; f_2),\nonumber 
 \end{eqnarray}
 where $\mathsf{inc}$ denotes the inclusion of spaces, accompanied by the obvious restrictions of functions and foliations. The symbol ``$\downarrow \; ?? \;\;$ " indicates the unknown maps in the diagrams.
 
 These two types of problems come in a big variety of flavors, depending on the more or less rich boundary data and on the anticipated quality of the transformations $\Phi$ (homeomorphisms, $\mathsf{PD}$-homeomorphisms, H\"{o}lder homeomorphisms with some control of the H\"{o}lder exponent, and diffeomorphisms with different degree of smoothness).
\smallskip

Let us formulate the main result of \cite{K4}, Theorem 4.1, which captures the philosophy of this article and puts our main result, Theorem \ref{th.main_alg}, in the proper context. Theorem \ref{th.main} reflects the scheme depicted in (\ref{3_Extension Holo}).

\begin{theorem}{\bf(Conjugate Holographic Extensions)}\label{th.main} Let $X_1, X_2$ be compact connected oriented smooth $(n+1)$-dimensional manifolds with boundaries. Consider two traversing boundary generic vector fields $v_1, v_2$ on $X_1$ and $X_2$, respectively. 
In addition, assume that $v_1, v_2$ have Property $\mathsf A$ from Definition \ref{def.property_A}.
\smallskip

Let a smooth orientation-preserving diffeomorphism $\Phi^\d: \d_1X_1 \to \d_1X_2$ commute with the two causality maps:
$$C_{v_2} \circ \Phi^\d = \Phi^\d \circ C_{v_1}$$

Then $\Phi^\d$ extends to a smooth orientation-preserving diffeomorphism $\Phi: X_1 \to X_2$ such that $\Phi$ maps the oriented foliation $\mathcal F(v_1)$ to the oriented foliation $\mathcal F(v_2)$. 
\end{theorem}

Let us outline the spirit of Theorem \ref{th.main}'s proof, since this will clarify the main ideas from Section 3. The reader interested in the technicalities may consult \cite{K4}. 

\begin{proof} 
First, using that $v_2$ is traversing, we construct a Lyapunov function $f_2: X_2 \to \R$ for $v_2$. 
Then we pull-back, via the diffeomorphism $\Phi^\d$, the restriction $f_2^\d := f_2|_{\d_1 X_2}$ to the boundary $\d_1 X_2$. Since  $\Phi^\d$ commutes with the two causality maps, the pull back $f_1^\d =_{\mathsf{def}} (\Phi^\d)^\ast(f^\d_2)$ has the property $f_1^\d(y) > f_1^\d(x)$ for any pair $y \succ x$ on the same $v_1$-trajectory, the order of points being defined by the $v_1$-flow. Equivalently, we get $f_1^\d(C_{v_1}(x)) > f_1^\d(x)$ for any $x \in \d_1^+X(v_1)$ such that $C_{v_1}(x) \neq x$. As the key step, we prove in \cite{K4} that such $f_1^\d$ extends to a smooth function $f_1: X_1 \to \R$ that has the property $df_1(v_1) > 0$. Hence, $f_1$ is a Lyapunov function for $v_1$. \smallskip

Recall that each causality map $C_{v_i}$, $i=1, 2$,  allows to view the $v_i$-trajectory space $\mathcal T(v_i)$ as the quotient space $(\d_1X_i)\big/ \{C_{v_i}(x) \sim x\}$, where $x \in \d_1^+X_i(v_i)$ and the topology in $\mathcal T(v_i)$ is defined as the quotient topology. Using that $\Phi^\d$ commutes with the causality maps $C_{v_1}$ and $C_{v_2}$, we conclude that $\Phi^\d$ induces a homeomorphism $\Phi^\mathcal T: \mathcal T(v_1) \to \mathcal T(v_2)$ of the trajectory spaces, which preserves their natural stratifications.\smallskip

For a traversing $v_i$, the manifold $X_i$ carries two mutually transversal  foliations: the oriented $1$-dimensional $\mathcal F(v_i)$, generated by the $v_i$-flow, and the foliation $\mathcal G(f_i)$, generated by the constant level hypersurfaces of the Lyapunov function $f_i$. To avoid dealing the singularities of $\mathcal F(v_i)$ and $\mathcal G(f_i)$, we extend $f_i$ to $\hat f_i: \hat X_i \to \R$ and $v_i$ to $\hat v_i$ on $\hat X_i$ so that $d\hat f_i(\hat v_i) > 0$. This generates nonsingular foliations $\mathcal F(\hat v_i)$ and $\mathcal G(\hat f_i)$ on $\hat X_i$. By this construction, $\mathcal F(\hat v_i)|_{\hat X_i} = \mathcal F(v_i)$ and $\mathcal G(\hat f_i)|_{X_i} = \mathcal G(f_i)$. 
Note that the ``leaves" of $\mathcal G(f_i)$ may be disconnected, while the leaves of $\mathcal F(v_i)$, the $v_i$-trajectories, are connected.  The two smooth foliations, $\mathcal F(\hat v_i)$ and $\mathcal G(\hat f_i)$, will serve as a ``coordinate grid" on $X_i$: every point $x \in X_i$ belongs to a \emph{unique} pair of leaves $\g_x \in \mathcal F(v_i)$ and $L_x := \hat f_i^{-1}(f_i(x)) \in \mathcal G(\hat f_i)$. 

Conversely, using the traversing nature of $v_i$, any pair $(y,\,  t)$, where $y \in \g_x \cap \d_1X_i$ and $t \in [f_i^\d(\g_x \cap \d_1X_i)] \subset \R$, where $[f_i^\d(\g_x \cap \d_1X_i)]$ denotes the minimal closed interval that contains the finite set $f_i^\d(\g_x \cap \d_1X_i)$, determines a \emph{unique} point $x \in X_i$. Note that some pairs of leaves $L$ and $\g$ may have an empty intersection, and some components of leaves $L$ may have an empty intersection with the boundary $\d_1X_i$.  

In fact, using that $f_i$ is a Lyapunov function, the hyprsurface $L = f_i^{-1}(c)$ intersects with a $v_i$-trajectory $\g$ if and only if $c \in [f_i^\d(\g \cap \d_1X_i)]$. Since the two smooth leaves, $\hat\g_y$ and $\hat f_i^{-1}(f_i(z))$, depend smoothly on the points $y, z \in \d_1X_i$ and are transversal, their intersection point $\hat\g_y \cap \hat f_i^{-1}(f_i(z)) \in \hat X_i$ depends smoothly on $(y, z) \in (\d_1X_i) \times (\d_1X_i)$, as long as $f_i^\d(z) \in [f_i^\d(\g_y \cap \d_1X_i)]$. 
Note that pairs $(y, z)$, where $y, z \in \d_1X_i$, with the property $f_i^\d(z) \in f_i^\d(\g_y \cap \d_1X_i)$ give rise to the intersections $\hat\g_y \cap \hat f_i^{-1}(f_i(z))$ that belong to $\d_1X_i$.

\smallskip

Now we are ready to extend the diffeomorphism $\Phi^\d$ to a homeomorphism $\Phi: X_1 \to X_2$. 
In the process, following the scheme in (\ref{3_Extension Holo}), \emph{we assume the the foliations $\mathcal F(v_i)$ and of the Lyapunov functions $f_i$ on $X_i$ ($i= 1, 2$) do exist and are ``knowable",  although we have access only to their traces on the boundaries}. 

Take any $x \in X_1$. It belongs to a unique pair of leaves $L_x  \in \mathcal G(f_1)$ and $\g_x \in \mathcal F(v_1)$. We define $\Phi(x) = x' \in X_2$, where $x'$ is the unique point that belongs to the intersection of $f_2^{-1}(f_1(x)) \in \mathcal G(f_2)$ and the $v_2$-trajectory $\g' = \Gamma_2^{-1}(\Phi^\mathcal T(\g_x))$. 
By its construction, $\Phi |_{\d_1X_1} = \Phi^\d$. Therefore, $\Phi$ induces the same homeomorphism $\Phi^{\mathcal T}: \mathcal T(v_1) \to \mathcal T(v_2)$ as $\Phi^\d$ does. 

The leaf-hypersurface $\hat f_2^{-1}(f_1(x))$ depends smoothly on $x$, but the leaf-trajectory $\hat \g' = \Gamma_2^{-1}(\Phi^\mathcal T(\hat \g_x))$ may not!
 Although the homeomorphism $\Phi$ is a diffeomorphism along the $v_1$-trajectories, it is not clear that it is a diffeomorphism on $X_1$ (a priori, $\Phi$ is just a H\"{o}lder map with a H\"{o}lder exponent $\a = 1/m$, where $m$ is the maximal tangency order of $\g$'s to $\d_1X$). Presently, for proving that $\Phi$ is a diffeomorphism, we need  Property $\mathsf A$ from Definition \ref{def.property_A}. Assuming its validity, we use the transversality of $\g_x$ \emph{somewhere} to $\d_1X$ to claim the smooth dependence of $\Gamma_2^{-1}(\Phi^\mathcal T(\hat \g_x))$ on $x$. Now, since the smooth foliations $\mathcal F(\hat v_i)$ and $\mathcal G(\hat f_i)$ are transversal, it follows that $x' = \Phi(x)$ depends smoothly on $x$. Conjecturally, Property $\mathsf A$ is unnecessary for establishing that $\Phi$ is a diffeomorphism. \end{proof}
 
Note that this construction of the extension $\Phi$ is quite explicit, but not canonic. For example, it depends on the choice of extension of $f_1^\d := (\Phi^\d)^\ast(f_2^\d)$ to a smooth function $f_1: X_1 \to \R$, which is strictly monotone along the $v_1$-trajectories. The uniqueness (topological rigidity) of the extension $\Phi$ may be achieved, if one assumes \emph{knowing fully} the manifolds $X_i$, equipped with the foliation grids $\mathcal F(v_i), \mathcal G(f_i)$ and the Lyapunov function $f_i$. In  Theorem \ref{th.main_alg},  we will reflect on this issue. 
\smallskip  

The next theorem (see \cite{K4}, Corollary 4.3) fits the scheme in (\ref{3_Reconstruction Holo}). It claims that the \emph{smooth topological type} of the triple $\{X, \mathcal F(v), \mathcal G(f)\}$ may be reconstructed from the appropriate boundary-confined data, provided that Property $\mathsf A$ is valid.

\begin{corollary}{\bf(Holography of Traversing Flows)}\label{cor.main_reconstruction} Let $X$ be a compact connected smooth $(n+1)$-dimensional manifold with boundary, and let $v$ be a traversing boundary generic vector field, which possesses  Property $\mathsf A$. 
\smallskip

Then the following boundary-confined data: 
\begin{itemize}
\item the causality map $C_v: \d_1^+X(v) \to \d_1^-X(v)$,
\item  the restriction $f^\d: \d_1 X \to \R$ of the Lyapunov function $f$,
\end{itemize}
are sufficient for  reconstructing the triple $(X, \mathcal F(v), f)$, 
up to  diffeomorphisms $\Phi: X \to X$ which are the identity on the boundary $\d_1 X$. 
\end{corollary}

\begin{proof} 
We claim that, in the presence of Property $\mathsf A$, the data $\{C_v,\; \; f^\d\}$ 
on the boundary $\d_1X$ allow for a reconstruction of the triple $(X, \mathcal F(v), f)$, up to a diffeomorphism that is the identity on $\d_1X$. 

Assume that there exist two traversing flows $(X_1, \mathcal F(v_1), f_1)$ and $(X_2, \mathcal F(v_2), f_2)$ such that $\d_1X_1 = \d_1X_2 = \d_1X$,  $$\{C_{v_1},\; f_1^\d\} = \{C_{v_2} ,\; \; f_2^\d\} = \{C_{v}, f^\d\}.$$  Applying Theorem \ref{th.main} to the identity diffeomorphism $\Phi^\d = \mathsf{id_{\d_1 X}}$, we conclude that it extends to a diffeomorphism $\Phi: X_1 \to X_2$ that takes $\{\mathcal F(v_1) \cap \d_1X_1,\; \; f_1^\d\}$ to $\{\mathcal F(v_2) \cap \d_1X_2,\; \; f_2^\d\}$.
\end{proof}

\begin{remark}
\emph{Unfortunately, Corollary \ref{cor.main_reconstruction} and its proof are not very constructive. They are just claims of existence: at the moment, it is not clear how to build the triple $(X, \mathcal F(v), f)$ only from the boundary data $(\d_1X, C_v, f^\d)$. \hfill $\diamondsuit$}
\end{remark}

\begin{figure}[ht]
\centerline{\includegraphics[height=2in,width=3in]{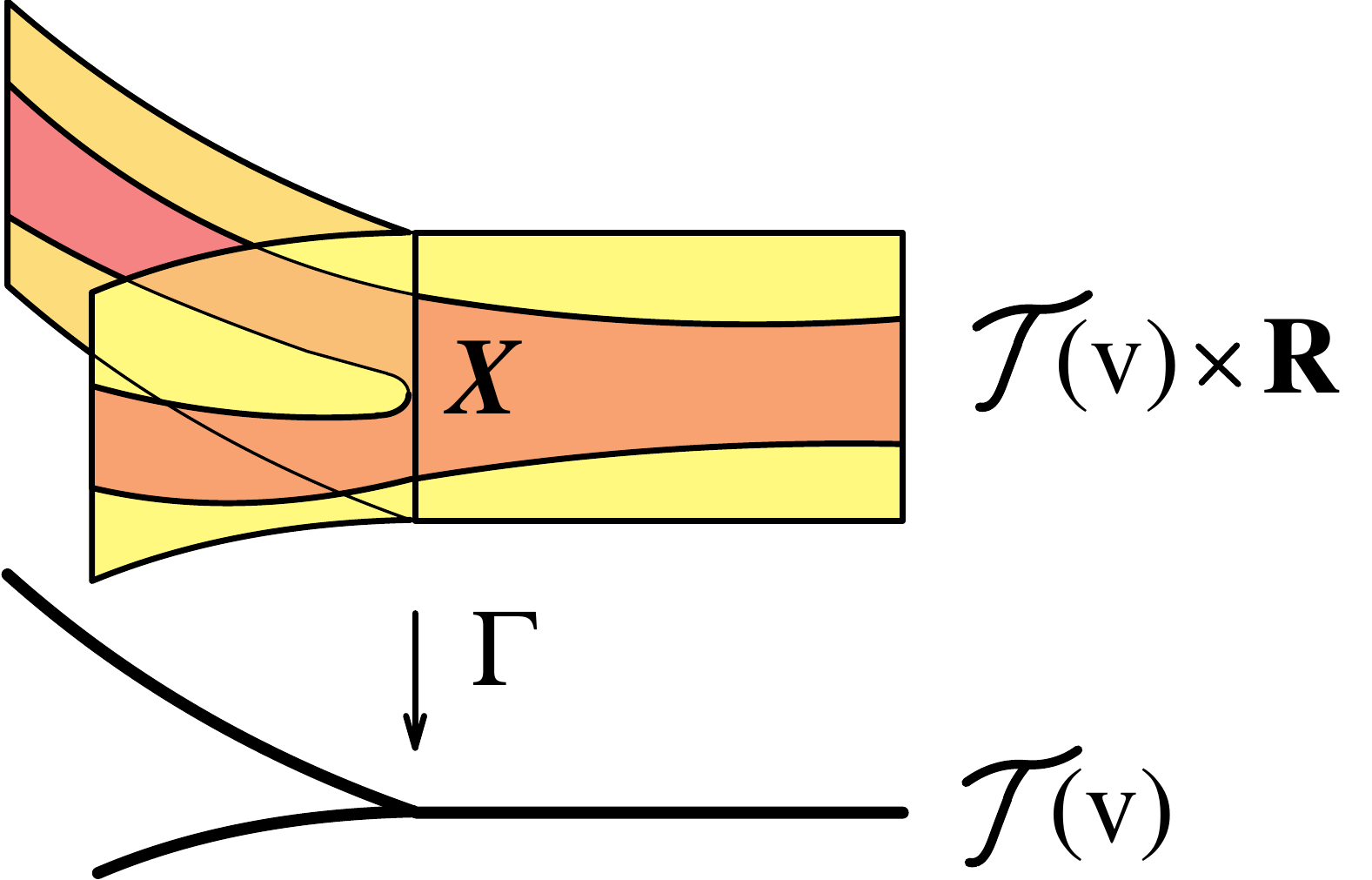}}
\bigskip
\caption{\small{Embedding $\a: X \to \mathcal T(v) \times \R$, produced by $\Gamma: X \to \mathcal T(v)$ and $f: X \to \R$.}}
\label{fig.AAAA}
\end{figure}

Fortunately, the following simple construction (\cite{K4}, Lemma 3.4), shown in Fig.\ref{fig.AAAA}, produces an explicit recipe for recovering the triple $(X, \mathcal F(v), f)$ from the triple $(\d_1X, C_v, f^\d)$, but only up to a \emph{homeomorphism}.  \smallskip


As we have seen in the proof of Theorem \ref{th.main}, the causality map $C_v$ determines the quotient trajectory space $\mathcal T(v)$ canonically. Let $f: X \to \R$ be a Lyapunov function for $v$.\smallskip

 The pair $(\mathcal F(v), f)$ gives rise to an embedding $\a: X \hookrightarrow \mathcal T(v) \times \R$, defined by the formula $\a(x) = ([\g_x], f(x))$, where $x \in X$ and $[\g_x] \in \mathcal T(v)$ denotes the point-trajectory through $x$. The dependece $x \leadsto [\g_x]$ is continuous by the definition of the quotient topology in $ \mathcal T(v)$.  Note that $\a$ maps each $v$-trajectory $\g$ to the line $[\g] \times \R$, and, for any $c \in \R$, each (possibly disconnected) leaf $\mathcal G_c :=  f^{-1}(c)$ to the slice $\mathcal T(v) \times c$ of $\mathcal T(v) \times \R$.  With the the help of the embedding $\a$, each trajectory $\g \in \mathcal F(v)$ may be identified with the closed interval $[f^\d(\g \cap \d_1X)] \subset \R$, and the vector field $v|_\g$ with the constant vector field $\d_u$ on $\R$.\smallskip

Consider now the restriction $\a^\d$ of the embedding $\a$ to the boundary $\d_1 X$. Evidently, the image of $\a^\d: \d_1 X \hookrightarrow \mathcal T(v) \times \R$ bounds  the image $\a(X) = \coprod_{[\g] \in \mathcal T(v)} [f^\d(\g \cap \d_1X)]$. Therefore, using the product structure in $\mathcal T(v) \times \R$, $\a^\d(\d_1 X)$ determines $\a(X)$ canonically. Hence, $\a(X)$ depends on $C_{v}$ and $f_1^\d$ only! Note that $\a$ is a continuous $1$-to-$1$ map on a compact space, and thus, a homeomorphism onto its image. Moreover, the topological type of $X$ depends only on $C_{v}$: the apparent dependence of $\a(X)$ on $f^\d$ is not crucial, since, for a given $v$, the space 
$\mathsf{Lyap}(v)$ of Lyapunov functions for $v$  is convex.\smallskip

The standing issue is: How to make sense of the claim ``$\a$ is a diffeomorphism"? Section 3 descibes  our attempt to address this question (see Lemma \ref{lem.3.8_iso} and Theorem \ref{th.main_alg}). 
\smallskip


\section{Recovering the algebra $C^\infty(X)$ in terms of subalgebras of $C^\infty(\d_1 X)$}

In what follows, we are inspired by the following classical property of functional algebras: for any compact smooth manifolds $X, Y$, we have an algebra isomorphism $C^\infty(X \times Y) \approx C^\infty(X) \hat\otimes C^\infty(Y)$, where $\hat \otimes$ denotes an appropriate completion of the algebraic tensor product $C^\infty(X)\otimes C^\infty(Y)$ \cite{Grot}.\smallskip

The trajectory space $\mathcal T(v)$, although a singular space, carries a surrogate smooth structure \cite{K3}. By definition, a function $h:  \mathcal T(v) \to \R$ is smooth if its pull-back $\Gamma^\ast(h): X \to \R$ is a smooth function on $X$. As a subspace of $C^\infty(X)$, the  $C^\infty(\mathcal T(v))$ is formed exactly by the smooth functions $g: X \to \R$, whose directional derivatives $\mathcal L_vg$ vanish in $X$. If $\mathcal L_v(g) = 0$ and $\mathcal L_v(h)= 0$, then $\mathcal L_v(g\cdot h) = \mathcal L_v(g) \cdot h + g\cdot \mathcal L_v(h)= 0$. Thus, $C^\infty(\mathcal T(v))$ is indeed a subalgebra of $C^\infty(X)$.
 
Note that if we change $v$ by a non-vanishing conformal factor $\l$, then $\mathcal L_vg =0$ if and only if $\mathcal L_{\l \cdot v}\, g =0$. Therefore, the algebra $C^\infty(\mathcal T(v))$ depends only on the {\sf conformal class} of $v$; in other words, on the foliation $\mathcal F(v)$.  

In the same spirit, we may talk about {\sf diffeomorphisms} $\Phi^\mathcal T: \mathcal T(v) \to  \mathcal T(v)$ of the trajectory spaces, as maps that induce isomorphisms of the algebra $C^\infty(\mathcal T(v))$. \smallskip

If two ($v$-invariant) functions from $C^\infty(\mathcal T(v))$ take different values at a point $[\g] \in \mathcal T(v)$, then they must take different values on the finite set $\g \cap \d_1X \subset \d_1 X$. Therefore, the obvious restriction homomorphism  $res_\mathcal T^\d: C^\infty(\mathcal T(v)) \to C^\infty(\d_1X)$, induced by the inclusion $\d_1X \subset X$, is a \emph{monomorphism}. We denote its image by $C^\infty(\d_1X, v)$. Thus, we get an isomorphism $res_\mathcal T^\d: C^\infty(\mathcal T(v)) \to C^\infty(\d_1X, v)$. We think of the subalgebra $C^\infty(\d_1X, v) \subset C^\infty(\d_1X)$ as an integral part of the boundary data for the holography problems we are tackling. \smallskip

Let $\pi_k: J^k(X, \R) \to X$ be the vector bundle of $k$-jets of smooth maps from $X$ to $\R$. We choose a continuous family semi-norms $|\sim |_k$ in the fibers of the jet bundle $\pi_k$ and use it to define a sup-norm $\|\sim\|_k$ for the sections of $\pi_k$.  We denote by $jet^k$ the obvious map $C^\infty(X, \R) \to J^k(X, \R)$ that takes each function $h$ to the collection of its $k$-jets $\{jet^k_x(h)\}_{x \in X}$. \smallskip

The {\sf Whitney topology}  \cite{W3} in the space $C^\infty(X) = \{h: X \to \R\}$ is defined in terms of the countable family of the norms $\{\|jet^k(h)\|_k\}_{k \in \N}$ of such sections $jet^k(h)$ of $\pi_k$.  
This topology insures the uniform convergence, on the compact subsets of $X$, of functions and their partial derivatives of an arbitrary order.  Note also that $\|jet^k(h_1 \cdot h_2)\|_k \leq \|jet^k(h_1)\|_k \cdot \|jet^k(h_2)\|_k$ for any $h_1, h_2 \in C^\infty(X)$. 

Any subalgebra $\mathcal A \subset C^\infty(X)$ inherits a topology from the Whitney topology in $C^\infty(X)$. In particular, the subalgebra $C^\infty(\mathcal T(v)) \approx C^\infty(X, v)$ does. \smallskip

As a locally convex vector spaces, 
$C^\infty(\mathcal T(v))$ and $C^\infty(\R)$ are then nuclear (\cite{DS}, \cite{Ga}) so that the {\sf topological tensor product} $C^\infty(\mathcal T(v))\,\hat\otimes \, C^\infty(\R)$ (over $\R$) is uniquely defined as the completion of the algebraic tensor product $C^\infty(\mathcal T(v))\, \otimes \, C^\infty(\R)$ \cite{Grot}. \smallskip
\smallskip

We interpret $C^\infty(\mathcal T(v))\,\hat\otimes\, C^\infty(\R)$ as the algebra of ``smooth" functions on the product $\mathcal T(v) \times \R$ and denote it by $C^\infty(\mathcal T(v) \times \R)$. \smallskip


\begin{lemma}\label{lem.A}
The intersection $C^\infty(\mathcal T(v)) \cap (f)^\ast(C^\infty(\R)) = \underline \R$, the space of constant functions on $X$. 
\end{lemma}

\begin{proof}
If a smooth function $h: X \to \R$ is constant on each $v$-trajectory $\g$ and belongs to $(f)^\ast(C^\infty(\R))$, then it must be constant on each connected leaf of $\mathcal G(f)$ that intersects $\g$. Thus, such $h$ is constant on the maximal closed \emph{connected} subset $A_\g \subseteq f^{-1}(f(\g))$ that contains $\g$. Each trajectory $\g$, homeomorphic to a closed interval, has an open neighborhood such that, for any trajectory $\g'$ from that neighborhood, we have $A_\g \cap A_{\g'} \neq \emptyset$. Since $X$ is connected, any pair $\g, \g'$ of trajectories may be connected by a path $\delta \subset X$. Using the compactness of $\delta$, we conclude that the function $h$ must be a constant along $\delta$. Therefore, $h$ is a constant globally. 
\end{proof}

Let us consider two subalgebras, $f^\ast(C^\infty(\R)) \subset C^\infty(X)$ and $(f^\d)^\ast(C^\infty(\R)) \subset C^\infty(\d_1X)$, the second one is assumed to be a ``known" part of the boundary data. 

\begin{lemma}\label{lem.3.7_iso} The restriction operator $H_f^\d: f^\ast(C^\infty(\R))  \to(f^\d)^\ast(C^\infty(\R))$ to the boundary $\d_1X$ is an epimorphism of algebras. If  the range $f^\d(\d_1X)$ of $f^\d$ is a connected closed interval of $\R$ (which is the case for a connected $\d_1X$), then $H_f^\d$ is an isomorphism.
\end{lemma}

\begin{proof} The restriction operator $H_f^\d$ is an algebra epimorphism,  
since any composite function $\phi \circ f^\d$, where $\phi \in C^\infty(\R)$, is the restriction to $\d_1X$ of the function $\phi \circ f$.

On the other hand, when $f^\d(\d_1X)$ is a connected subset of $\R$, we claim that $H_f^\d$ is a monomorphism. Indeed, take a function $\phi \in  C^\infty(\R)$, such that $\phi \circ f^\d \equiv 0$, but $\phi \circ f$ is not identically zero on $X$. Then there is $x \in X$ such that  $\phi \circ f(x) \neq 0$.
On the other hand, by the hypothesis, $f(x)= f^\d(y)$ for some $y\in \d_1X$. By the assumption, $f^\d \circ \phi \equiv 0$, which implies that $\phi(f^\d(y)) = 0$. This contradiction validates the claim about $H_f^\d$ being a monomorphism. Therefore, when $f^\d(\d_1X)$ is a connected interval,  $H_f^\d$ is an isomorphism of algebras.
\end{proof}


Consider the homomorphism of algebras $$\mathsf P: C^\infty(\mathcal T(v)) \otimes f^\ast(C^\infty(\R)) \to C^\infty(X)$$ 
 that takes every finite sum $\sum_i h_i \otimes (f \circ g_i)$, where $h_i \in C^\infty(\mathcal T(v)) \subset C^\infty(X)$ and $g_i \in C^\infty(\R)$, to the finite sum $\sum_i h_i \cdot (g_i \circ f) \in C^\infty(X)$. 
 
Recall that, by Lemma \ref{lem.A}, $C^\infty(\mathcal T(v)) \cap (f)^\ast(C^\infty(\R)) = \underline \R$, the constants. For any linearly independent $\{h_i\}_i$, this lemma implies that if $\sum_i h_i \cdot (g_i \circ f) \equiv 0$, then $\{g_i \circ f \equiv 0\}_i$; therefore, $\mathsf P$ is a monomorphism.\smallskip 

Let us compare the, so called, {\sf projective crossnorms} $\{\|\sim\|_{k}\}_{k \in \Z_+}$ (see (\ref{eq.norms_A})) of an element $$\phi = \sum_i h_i \otimes (f \circ g_i)$$  and the norms of the element $\mathsf P(\phi) = \sum_i h_i \cdot (f \circ g_i)$. By comparing the Taylor polynomial of the product of two smooth functions with the product of their Taylor polynomials, we get that, for all $k\in \Z_+$, 
\begin{eqnarray}
\|\phi\|_{k} =_{\mathsf{def}}\; \mathsf{\inf}\big\{\sum_i \|h_i\|_k \cdot \|(f \circ g_i)\|_k\big\}     \geq \;  \|\mathsf P(\phi)\|_{k}, \label{eq.norms_A}
\end{eqnarray} 
where $\mathsf{\inf}$ is taken over all the representations of the element $\phi \in C^\infty(\mathcal T(v)) \otimes f^\ast(C^\infty(\R))$ as a sum $\sum_i h_i \otimes (f \circ g_i)$. Here we may assume that all $\{h_i\}_i$  are linearly independent elements and so are all $\{f \circ g_i\}_i$; otherwise, a simpler representation of $\phi$ is available.  \smallskip

By the inequality in (\ref{eq.norms_A}),  $\mathsf P$ is a bounded (continuous) operator. As a result, by continuity, $\mathsf P$ extends to an algebra homomorphism 
$$\hat{\mathsf P}: C^\infty(\mathcal T(v))\, \hat\otimes\, f^\ast(C^\infty(\R)) \to C^\infty(X)$$ 
whose source is the completion of the algebraic tensor product $C^\infty(\mathcal T(v)) \otimes f^\ast(C^\infty(\R))$.

\begin{lemma}\label{lem.3.8_iso} 
The embedding $\a: X \to \mathcal T(v) \times \R$ (introduced in the end of Section 2 and depicted in Fig. \ref{fig.AAAA}) induces an algebra epimorphism  
\begin{eqnarray}\label{eq.alpha_alg}
\a^\ast: C^\infty(\mathcal T(v))\,\hat\otimes\, C^\infty(\R) \stackrel{\mathsf{id}\, \hat\otimes\, f^\ast}{\longrightarrow} C^\infty(\mathcal T(v))\,\hat\otimes\, f^\ast(C^\infty(\R)) \stackrel{\hat{\mathsf P}}{\longrightarrow} C^\infty(X).  
\end{eqnarray}
Moreover, the map $\hat{\mathsf P}$ is an \emph{isomorphism}. 
\end{lemma}

\begin{proof}  First, we claim that the subalgebra $\mathsf P\big(C^\infty(\mathcal T(v))\otimes f^\ast(C^\infty(\R))\big) \subset C^\infty(X)$ satisfies the three hypotheses of Nachbin's Theorem \cite{Na}. Therefore, by \cite{Na}, the $\mathsf P$-image of $C^\infty(\mathcal T(v))\otimes f^\ast(C^\infty(\R))$ is \emph{dense} in $C^\infty(X)$.  Let us validate these three hypotheses.  
\begin{enumerate}
\item \emph{For each $x \in X$, there is a function $q \in C^\infty(\mathcal T(v))\otimes f^\ast(C^\infty(\R))$ such that $q(x) \neq 0$.}\smallskip

Just take $q =  f \circ (t + c)$, where $c > \min_X f$ and $t: \R \to \R$ is the identity. \smallskip

\item \emph{For each $x, y \in X$, there is a function $q \in C^\infty(\mathcal T(v))\otimes f^\ast(C^\infty(\R))$ such that $q(x) \neq q(y)$ (i.e., the algebra $C^\infty(\mathcal T(v))\otimes f^\ast(C^\infty(\R)))$ separates the points of $X$)} \smallskip

If $f(x) \neq f(y)$, $q = f$ will do. If $f(x) = f(y)$, but $[\g_x] \neq [\g_y]$, then there is a $v$-invariant function $h \in C^\infty(\mathcal T(v))$ such that $h(x) =1$ and $h(y) = 0$.  To construct this $h$, we take a transversal section $S_x \subset \hat X$ of the $\hat v$-flow in the vicinity of $x$  such that all the $\hat v$-trajectories through $S_x$ are distinct from the trajectory $\g_y$. We pick a smooth function $\tilde h: S_x \to \R$ such that $h$ is supported in $\mathsf{int}(S_x)$, vanishes with all its derivatives along the boundary $\d S_x$, and  $\tilde h(x) =1$. Let  $\mathcal S$ denote the set of $\hat v$-trajectories through $S_x$. Of course, $\tilde h$ extends to a smooth function $h^\dagger: \mathcal S \to \R$ so that $h^\dagger$ is constant along each trajectory from $\mathcal S$. We denote by $h^\ddagger$  the obvious extension of $h^\dagger$ by the zero function. Finally, the restriction $h$ of $h^\ddagger$ to $X$ separates $x$ and $y$. 
\smallskip

\item \emph{For each $x \in X$ and $w \in T_xX$,  there is a function $q \in C^\infty(\mathcal T(v))\otimes f^\ast(C^\infty(\R))$ such that $dq_x(w) \neq 0$. }

Let us decompose $w = av +bw^\dagger$, where $a, b \in \R$ and the vector $w^\dagger$ is tangent to the hypersurfce $S_x = \hat f^{-1}(f(x))$. Then, if $a \neq 0$, then $df(w) \neq 0$. If $a=0$, then the there is a function $\tilde h: S_x \to \R$ which, with all its derivatives, is compactly supported  in the vicinity of $x$ in $S_x$ and  such that $d\tilde h_x(w^\dagger) \neq 0$.  As in the case (2), this function extends to a desired function $h \in C^\infty(\mathcal T(v))$. Now put $q = h \otimes 1$.
\end{enumerate}
\smallskip

As a result, the image of $\mathsf P: C^\infty(\mathcal T(v))\otimes f^\ast(C^\infty(\R)) \longrightarrow C^\infty(X)$  is dense. Therefore,  $\hat{\mathsf P}$ and, thus, $(\a)^\ast: C^\infty(\mathcal T(v))\, \hat \otimes \, C^\infty(\R) \longrightarrow C^\infty(X)$ are epimorphisms. 
\smallskip

Let us show that $\hat{\mathsf P}$ is also a monomorphism. Take a typical element $$\theta = \sum_{i=1}^\infty h_i \otimes (f \circ g_i) \in C^\infty(X, v) \,\hat\otimes\, f^\ast(C^\infty(\R)),$$ viewed as a sum that converges in all the norms $\|\sim \|_k$ from (\ref{eq.norms_A}).   We aim to prove that if $\hat{\mathsf P}(\theta) =  
\sum_{i=1}^\infty h_i \cdot (f \circ g_i)$ vanishes on $X$, then $\theta = 0$.

For each point $x \in \mathsf{int}(X)$, there is a small closed cylindrical solid $H_x \subset \mathsf{int}(X)$ that contains $x$ and consists of segments of trajectories through a small $n$-ball $D^n \subset f^{-1}(f(x))$, transversal to the flow. Thus, the product structure $D^1 \times D^n$ of the solid $H_x$ is given by the $v$-flow and the Lyapunov function $f: X \to \R$. 

We localize the problem to the cylinder $H_x$.
Consider the commutative diagram 
\begin{eqnarray}\label{eq.DIAGRAM}
C^\infty(X, v) \,\hat\otimes\, f^\ast(C^\infty(\R)) \stackrel{\hat{\mathsf P}}{\longrightarrow} C^\infty(X). \nonumber \\
\downarrow \mathsf{res'} \hat\otimes \mathsf{res''}   \quad \quad \quad \quad \quad \quad \quad \downarrow \mathsf{res} \nonumber \\
C^\infty(D^n)\,\hat\otimes\, C^\infty(D^1) \stackrel{\approx\;\hat{\mathsf Q}}{\longrightarrow} C^\infty(H_x),
\end{eqnarray}
where $\mathsf{res} : C^\infty(X) \to C^\infty(H_x)$ is the natural homomorphism,
$$(\mathsf{res'} \hat\otimes \mathsf{res''})\big(\sum_{i=1}^\infty h_i \otimes (f \circ g_i)\big) =_{\mathsf{def}}\; \sum_{i=1}^\infty h_i|_{D^n} \otimes (f \circ g_i)|_{D^1},$$
and $\hat{\mathsf Q}\big(\sum_{i=1}^\infty \tilde h_i \otimes \tilde{g_i}\big) =_{\mathsf{def}}\;  \sum_{i=1}^\infty \tilde h_i \cdot \tilde{g_i}$
\; for  $\tilde h_i \in C^\infty(D^n),\, \tilde g_i \in C^\infty(D^1)$. \smallskip

Since $\hat{\mathsf Q}$ is an isomorphism \cite{Grot} and $\hat{\mathsf  P}(\theta) = 0$, it follows from (\ref{eq.DIAGRAM}) that $\theta \in \ker(\mathsf{res'} \hat\otimes \mathsf{res''})$ for any cylinder $H_x$. After reshuffling terms in the sum, one may assume that all the functions $\{h_i|_{D^n} \}_i$ are linearly independent. Using that the functions $h_i|_{D^n}$ and $(f \circ g_i)|_{D^1}$ depend of the complementary groups of coordinates in $H_x$, we conclude that these functions must vanish for any $H_x \subset \mathsf{int}(X)$. As a result, $\theta = 0$ globally in $\mathsf{int}(X)$ and, by continuity, $\theta$ vanishes on $X$.
\hfill
\end{proof}

Consider now the ``known" homomorphism of algebras
\begin{eqnarray}\label{eq.alpha_d_alg}
(\a^\d)^\ast:\; C^\infty(\mathcal T(v))\,\hat\otimes\, C^\infty(\R) \stackrel{\mathsf{\approx res_\mathcal T^\d}\, \hat\otimes\, (f^\d)^\ast}{\longrightarrow} \nonumber \\
\longrightarrow C^\infty(\d_1X, v)\, \hat\otimes \, (f^\d)^\ast(C^\infty(\R))
 \stackrel{\hat{\mathsf{R}}^\d}{\longrightarrow} C^\infty(\d_1X),
\end{eqnarray}
utilizing the boundary data.
Here, by the definition of $C^\infty(\d_1X, v)$,  
$\mathsf{res}^\d_{\mathcal T} : C^\infty(\mathcal T(v)) \to C^\infty(\d_1X, v)$ is an isomorphism,  and 
$\hat{\mathsf{R}}^\d$ denotes the completion of the bounded homomorphism $\mathsf{R}^\d$ that takes each element $\sum_i h_i \otimes (f^\d \circ g_i)$, where $h_i \in C^\infty(\d_1X, v)$ and $g_i \in C^\infty(\R)$, to the sum $\sum_i h_i \cdot (g_i \circ f^\d)$.\smallskip

The next lemma shows that the hypotheses of Theorem \ref{th.main_alg} are not restrictive, even when $\d_1 X$ has many connected components.

\begin{lemma} Any traversing vector field $v$ on a connected compact manifold $X$ admits a Lyapunov function $f: X \to \R$ such that $f(X) = f(\d_1X)$.
\end{lemma}

\begin{proof} 
Note that, for any Lyapunov function $f$,  the image $f(\d_1X)$ is a disjoint union of finitely many closed intervals $\{I_k = [a_k, b_k]\}_k$, where the index $k$ reflects the natural order of intervals in $\R$. We will show how to decrease, step by step, the number of these intervals by deforming the original function $f$. Note that the local extrema of any Lyapunov function on $X$ occur on its boundary $\d_1X$ and away from the locus $\d_2X(v)$ where $v$ is tangent to $\d_1X$. Consider a pair of points $A_{k+1}, B_k \in \d_1X \setminus \d_2X(v)$ such that $f(A_{k+1})= a_{k+1}$ and $f(B_{k})= b_{k}$, where $a_{k+1} > b_{k}$. Then we can increase $f$ in the vicinity of its local maximum $B_k$ so that the $B_k$-localized deformation $\tilde f$ of $f$ has the property $\tilde f(B_k) > f(A_{k+1})$ and $\tilde f$ is a Lyapunov function for $v$. This construction decreases the number of intervals in $\tilde f(\d_1X)$ in conparison to $f(\d_1X)$ at least by one. \hfill
\end{proof}

We are ready to state the main result of this paper.

\begin{theorem}\label{th.main_alg} Assuming that the range $f^\d(\d_1X)$ is a connected interval of $\R$,\footnote{which is the case for a connected $\d_1X$} the algebra $C^\infty(X)$ is isomorphic to the subalgebra  $$C^\infty(\d_1X, v)\, \hat\otimes \, (f^\d)^\ast(C^\infty(\R)) \subset C^\infty(\d_1X)  \hat\otimes C^\infty(\d_1X).$$

Moreover,  by combining (\ref{eq.alpha_alg}) with (\ref{eq.alpha_d_alg}), we get a commutative diagram
\begin{eqnarray}\label{eq.DIAGRAM}
C^\infty(\mathcal T(v))\,\hat\otimes\, f^\ast(C^\infty(\R))  \stackrel{\hat{\mathsf{R}}}{\longrightarrow} C^\infty(X) \nonumber \\
\downarrow \mathsf{id}\, \hat\otimes\, H_f^\d  \quad \quad \quad \quad  \quad \quad \quad \quad    \downarrow  \mathsf{res} \; \nonumber \\
C^\infty(\d_1X, v)\, \hat\otimes \, (f^\d)^\ast(C^\infty(\R))
 \stackrel{\hat{\mathsf{R}}^\d}{\longrightarrow} C^\infty(\d_1X),
\end{eqnarray} 
whose vertical homomorphism $\mathsf{id}\, \hat\otimes\, H_f^\d $ and the horizontal homomorphism $\hat{\mathsf{R}}$ are isomorphisms, and the vertical epimorphism $\mathsf{res}$ is the obvious restriction operator.\smallskip

As a result, inverting $\mathsf{id}\, \hat\otimes\, H_f^\d$, we get an algebra isomorphism 
\begin{eqnarray}\label{eq.3.4_A}
\mathcal H(v, f): C^\infty(\d_1X, v)\, \hat\otimes \, (f^\d)^\ast(C^\infty(\R)) \approx C^\infty(X).
\end{eqnarray}
\end{theorem}

\begin{proof} Consider the commutative diagram (\ref{eq.DIAGRAM}).
Its upper-right conner is ``unknown", while the lower row is ``known" and represents the boundary data, and $\mathsf{res}$ is obviously an epimorphism. By Lemma \ref{lem.3.7_iso}, the left vertical arrow $\mathsf{id}\, \hat\otimes\, H_f^\d$ is an isomorphism. Since, by Lemma \ref{lem.3.8_iso}, $\hat{\mathsf{R}}$ is an isomorphism, it follows that $\hat{\mathsf{R}} \circ (\mathsf{id}\, \hat\otimes\, H_f^\d)^{-1}$ must be an isomorphism as well. In particular, $\hat{\mathsf{R}}^\d$ is an epimorphism, whose kernel is isomorphic to the smooth functions on $X$ whose restrictions to $\d_1X$ vanish. If $z \in C^\infty(X)$ is a smooth function such that zero is its regular value, $z^{-1}(0) = \d_1X$, and $z >0$ in $\mathsf{int}(X)$, then the kernel of $\mathsf{res}$ is the principle ideal $\mathsf m(z)$, generated by $z$. Therefore, by the commutativity of (\ref{eq.DIAGRAM}), the kernel of the homomorphism $\hat{\mathsf{R}}^\d$ must be also a principle ideal $\mathsf M_\d$, generated by an element  $(\hat{\mathsf{R}} \circ (\mathsf{id}\, \hat\otimes\, H_f^\d))^{-1}(z)$.
\end{proof}

\begin{corollary}\label{cor.3.1} If the range $f^\d(\d_1X)$ is a connected interval in $\R$, then the two   topological algebras $C^\infty(\d_1X, v) \subset C^\infty(\d_1X)$ and $(f^\d)^\ast(C^\infty(\R)) \subset C^\infty(\d_1X)$ determine, up to an isomorphism, the  algebra $C^\infty(X)$, and thus determine the smooth topological type of the manifold $X$. 
\end{corollary}

\begin{proof} We call a maximal ideal of an algebra $\mathcal A$ {\sf nontrivial} if it is different from $\mathcal A$. 

By Theorem \ref{th.main_alg}, the algebra $C^\infty(X)$ is determined by the two algebras on $\d_1X$, up to an isomorphism. In turn, the algebra $C^\infty(X)$ determines the smooth topological type of $X$, viewed as a ringed space. This fact is based on interpreting $X$ as the space $\mathcal M(C^\infty(X))$ of nontrivial maximal ideals of the algebra $C^\infty(X)$ \cite{KMS}.  

Let $\mathsf m^\d_v\, \triangleleft \;  C^\infty(\d_1X, v)$ and $\mathsf m^\d_f \, \triangleleft \;(f^\d)^\ast(C^\infty(\R))$ be a pair of nontrivial maximal ideals. Note that $\mathsf m^\d_v = \mathsf m^\d_v([\g])$ consists of functions from $C^\infty(\d_1X, v)$ that vanish on the locus $\g \cap \d_1X$,  and $\mathsf m^\d_f = \mathsf m^\d_f(c)$ consists of functions from $(f^\d)^\ast(C^\infty(\R))$ that vanish on the locus $\d_1X \cap f^{-1}(c)$, where $c \in f(\d_1X) \subset \R$.  We denote by $\langle \mathsf m^\d_v, \mathsf m^\d_f \rangle$ the maximal ideal of $C^\infty(\d_1X, v)\, \hat\otimes \,  (f^\d)^\ast(C^\infty(\R))$ that contains both ideals $\mathsf m^\d_v \hat\otimes 1$ and $1 \hat\otimes \mathsf m^\d_f$. If the range $f^\d(\d_1X)$ is a connected interval of $\R$ and $\langle \mathsf m^\d_v, \mathsf m^\d_f \rangle$ is a nontrivial ideal, then $\g \cap f^{-1}(c) \neq \emptyset$. Otherwise,  $\g \cap f^{-1}(c) = \emptyset$. Therefore, with the help of the isomorphism $\mathcal H(v, f)$ from (\ref{eq.3.4_A}), the nontrivial maximal ideals of $C^\infty(X)$ (which by \cite{KMS} correspond to points $x = \g \cap f^{-1}(c) \in X$) are of the form $\mathcal H(v, f)\big(\langle \mathsf m^\d_v, \mathsf m^\d_f \rangle\big)$.
\end{proof}

\begin{corollary} Let the range $f^\d(\d_1X)$ be a connected interval of $\R$. With the isomorphism $\mathcal H(v, f)$ from (\ref{eq.3.4_A}) being fixed, any algebra isomorphism $\Psi^\d: C^\infty(\d_1X) \to C^\infty(\d_1X)$  that preserves the subalgebras $C^\infty(\d_1X, v)$ and $(f^\d)^\ast(C^\infty(\R))$ extends canonically to the algebra isomorphism $\Psi: C^\infty(X) \to C^\infty(X)$. 

Thus, an action of any group $\mathsf G$ of such isomorphisms $\Psi^\d$ extends canonically to a $\mathsf G$-action on the algebra $C^\infty(X)$ and, via it, to a $\mathsf G$-action on $X$ by smooth diffeomorphisms.
\end{corollary}

\begin{proof} By \cite{Mr}, any algebra isomorphism $\Psi: C^\infty(X_1) \to  C^\infty(X_2)$ is induced by a unique smooth diffeomorphism $\Phi: X_1 \to X_2$. With this fact in hand, by Theorem \ref{th.main} and Theorem \ref{th.main_alg}, the proof is on the level of definitions.
\end{proof}

 It remains to address the following crucial  question:  how to characterize intrinsically the trace $C^\infty(\d_1X, v)$ of the algebra $C^\infty(\mathcal T(v)) \approx \ker\{\mathcal L_v: C^\infty(X) \to C^\infty(X)\}$ in the algebra $C^\infty(\d_1X)$?
\smallskip

Evidently,  functions from $C^\infty(\d_1X, v)$ are constant along each $C_v$-``trajectory" $\g^\d := \g \cap \d_1X$ of the causality map. Furthermore, any smooth function $\psi: \d_1X \to \R$ that is constant on each finite set $\g^\d$ gives rise to a unique \emph{continuous} function $\phi$ on $X$ that is constant along each $v$-trajectory $\g$. However, such functions $\phi$ may not be automatically \emph{smooth} on $X$ (a priory, they are just H\"{o}lderian with some control of the H\"{o}lder exponent that depends on the dimension of $X$ only)! This potential complication leads to the following question.


\begin{question}\label{q.3.1} For a traversing and boundary generic (alternatively, traversally generic) vector field $v$ on $X$, is it possible to characterize the subalgebra  $C^\infty(\d_1X, v)   \subset C^\infty(\d_1 X)$ in terms of the causality map $C_v$ and, perhaps, some additional $v$-generated data, residing in $\d_1 X$? \hfill $\diamondsuit$
\end{question}

To get some feel for a possible answer, we need the notion of the {\sf Morse stratification} of the boundary $\d_1X$ that a vector field $v$ generates \cite{Mo}. 

Let $\dim(X) = n+1$ and $v$ be a boundary generic traversing vector field on $X$. 

Let us recall the definition of the Morse stratification $\{\d_j^\pm X(v)\}_{j \in [1, n+1]}$ of $\d_1X$. We define the set $\d_2X(v)$ as the locus where $v$ is tangent to $\d_1X$. It separates $\d_1X$ into $\d_1^+X(v)$ and $\d_1^+X(v)$. Let $\d_3X(v)$ be the locus  where $v$ is tangent to $\d_2X(v)$. For a boundary generic $v$, $\d_2X(v)$ is a smooth submanifold of $\d_1X(v)$ and $\d_3X(v)$ is a submanifold that divides $\d_2X(v)$ into two regions, $\d_2^+X(v)$ and $\d_2^-X(v)$. Along $\d_2^+X(v)$, $v$ points inside of $\d_1^+X(v)$, and along $\d_2^-X(v)$, $v$ points inside of $\d_1^-X(v)$.  This construction self-replicates until we reach finite sets $\d_{n+1}^\pm X(v)$. 

By definition, the {\sf boundary generic} vector fields \cite{K1} are the ones that satisfy certain nested transversality of $v$ with respect to the boundary $\d_1X$, the transversality that guaranties that all the Morse strata $\d_jX(v)$ are regular closed submanifolds and  all the strata $\d_j^\pm X(v)$ are compact submanifolds.

For a traversing boundary generic $v$, the map $C_v: \d_1^+X(v) \to \d_1^-X(v)$ makes it possible to recover the Morse stratification $\{\d_j^\pm X(v)\}_{j>0}$  (\cite{K4}).

\smallskip

Let us describe now a good candidate for the subalgebra $C^\infty(\d_1X, v)$ in the algebra $C^\infty(\d X)$. 

We denote by $\mathcal L_{v}^{(k)}$ the $k$-th iteration of the Lie derivative $\mathcal L_{v}$. Let $\mathsf M(v)$ be the subalgebra of smooth functions $\psi: \d_1 X \to \R$  such that  $(\mathcal L_{v}^{(k)}\psi) \big |_{\d_{k+1}X(v)} = 0$ for all $k \leq n$ (by the Leibniz rule, $\mathsf M(v)$ is indeed a subalgebra). 
Let us denote by $\mathsf M(v)^{C_v}$ the subalgebra of functions from $\mathsf M(v)$ that are constant on each (finite) $C_v$-trajectory $\g^\d := \g \cap \d_1X \subset \d_1 X$.

\begin{conjecture}\label{conj.3.1} 
Let $v$ be a traversing and boundary generic vector field on a smooth compact $(n+1)$-manifold $X$. 
Then  the algebra $C^\infty(\d_1X, v)$ coincides with the subalgebra $\mathsf M(v)^{C_v} \subset C^\infty(\d_1 X)$. 

In particular, $C^\infty(\d_1X, v)$ can be determined by the causality map $C_v$ and the restriction of $v$ to $\d_2X(v)$.
\hfill $\diamondsuit$
\end{conjecture}

It is easy to check that $C^\infty(\d_1X, v) \subset \mathsf M(v)^{C_v}$; the challenge is to show that the two algebras coincide.
\smallskip

The Holography Theorem (Corollary \ref{cor.main_reconstruction}) has been established assuming Property {\sf A} from Definition \ref{def.property_A}. If one assumes the validity of Conjecture \ref{conj.3.1}, then, by Corollary \ref{cor.3.1}, we may drop Property {\sf A} from the hypotheses of the Holography Theorem. Indeed, the subalgebras $C^\infty(\d_1X, v)$ and $(f^\d)^\ast(C^\infty(\R))$ would acquire a description in terms of $C_v$ and $f^\d$. This would deliver an independent proof of a natural generalization of Corollary \ref{cor.main_reconstruction}.
\bigskip

{\it Acknowledgments:} The author is grateful to Vladimir Goldshtein for his valuable help with the analysis of spaces of smooth functions.

\end{document}